\documentclass{amsart}%
\pdfoutput=1
\usepackage{amssymb}
\usepackage{amsfonts}
\usepackage{amsmath}
\usepackage[usenames,dvipsnames]{xcolor}   
\usepackage{graphicx}%

\usepackage{hyperref}
\definecolor{darkblue}{rgb}{0.0,0.0,0.65}
\definecolor{darkred}{rgb}{0.68,0.05,0.0}
\definecolor{darkgreen}{rgb}{0.0,0.29,0.29}
\definecolor{darkpurple}{rgb}{0.47,0.09,0.29}

\hypersetup{
   colorlinks = true,
   citecolor  = darkblue,
   linkcolor  = darkred,
   filecolor  = darkblue,
   urlcolor   = darkblue,
 }

\usepackage{bm}  
\newcommand{\mbf}[1]{\bm{#1}}

\providecommand{\U}[1]{\protect\rule{.1in}{.1in}}
\newtheorem{theorem}{Theorem}
\theoremstyle{plain}

\newtheorem{conjecture}{Conjecture}
\newtheorem{corollary}{Corollary}

\newtheorem{definition}{Definition}
\newtheorem{example}{Example}

\newtheorem{lemma}{Lemma}

\newtheorem{problem}{Problem}

\newtheorem{remark}{Remark}

\numberwithin{equation}{section}

\DeclareMathOperator{\trace}{tr}

\begin{document}
\title[ ]{The Hornich-Hlawka functional inequality for functions with positive differences}
\author{Constantin P. Niculescu}
\address{Department of Mathematics, University of Craiova, Craiova 200585, Romania}
\email{constantin.p.niculescu@gmail.com}
\author{Suvrit Sra}
\address{Massachusetts Institute of Technology, Cambridge, MA 02139, USA}
\email{suvrit@mit.edu}
\dedicatory{Version 2}\date{January 19, 2023}

\begin{abstract}
We analyze the role played by $n$-convexity for the fulfillment of a series of
linear functional inequalities that extend the Hornich-Hlawka functional
inequality, $f\left(  x\right)  +f\left(  y\right)  +f\left(  z\right)
+f\left(  x+y+z\right)  \geq f\left(  x+y\right)  +f\left(  y+z\right)
+f\left(  z+x\right)  +f(0),$ including extensions to the case of positive operators.
\end{abstract}
\maketitle

\section{Introduction}
Many noteworthy inequalities are related to the following problem:

\begin{problem}
Suppose that $\mathcal{S}$ is an abelian additive semigroup with neutral
element $0,$ $f$ a function defined on $\mathcal{S}$ and taking values in an
ordered vector space $E$ $($or in its positive cone $E_{+}).$ For $n\geq2,$
find the linear inequalities relating%
\[
\sum\nolimits_{i=1}^{n}f(x_{i}),\text{ }\sum\nolimits_{1\leq i<j\leq n}%
f(x_{i}+x_{j}),\ldots,\text{ }f(\sum\nolimits_{i=1}^{n}x_{i})
\]
for all $x_{1},\ldots,x_{n}\in\mathcal{S}.$
\end{problem}

Due to its many ramifications, this problem is still the subject of intense
activity, and the present paper reports some new results in this direction, under the umbrella of the \emph{Hornich-Hlawka functional inequality.} Specifically, we are interested in studying
conditions under which a continuous function $f:\mathcal{S}\rightarrow
\mathbb{R}_{+}$ satisfies
\begin{equation}
f\left(  x\right)  +f\left(  y\right)  +f\left(  z\right)  +f\left(
x+y+z\right)  \geq f\left(  x+y\right)  +f\left(  y+z\right)  +f\left(
z+x\right)  \label{HHfunct_ineq+}%
\end{equation}
for all $x,y,z\in\mathcal{S}$. When $f$ is a real-valued function it is usual
to replace (\ref{HHfunct_ineq+}) by
\begin{equation}
f\left(  x\right)  +f\left(  y\right)  +f\left(  z\right)  +f\left(
x+y+z\right)  \geq f\left(  x+y\right)  +f\left(  y+z\right)  +f\left(
z+x\right)  +f(0). \label{HHfunct_ineq}%
\end{equation}

A good start for understanding the Hornich-Hlawka functional inequality is
provided by the following elementary (but powerful) inequality:%
\begin{equation}
|x|+|y|+|z|+|x+y+z|\geq|x+y|+|y+z|+|z+x|\text{\quad for all }x,y,z\in
\mathbb{R}. \label{ineqHHelem}%
\end{equation}
As was noticed by Levi \cite{Levi}, every piecewise linear inequality like
(\ref{ineqHHelem}) remains true when the real variables $x,y,z$ are replaced
by arbitrary vectors $\mbf{x},$ $\mbf{y},$ $\mbf{z}$ in $\mathbb{R}^{N}$ and the absolute value function is replaced
by the Euclidean norm,%
\begin{equation}
\left\Vert \mbf{x}\right\Vert +\left\Vert \mbf{y}\right\Vert +\left\Vert
\mbf{z}\right\Vert +\left\Vert \mbf{x}+\mbf{y}+\mbf{z}\right\Vert
\geq\left\Vert \mbf{x}+\mbf{y}\right\Vert +\left\Vert \mbf{y}%
+\mbf{z}\right\Vert +\left\Vert \mbf{z}+\mbf{x}\right\Vert .
\label{ineq_HH}%
\end{equation}
Inequality~\eqref{ineq_HH} is what is nowadays known as the \emph{Hornich-Hlawka inequality}. See
the paper of Hornich~\cite{Ho1942}, which includes the marvelous argument of
Hlawka, based on the triangle inequality and an identity (due to Fr\'{e}chet
\cite{Fr1935}) which characterizes inner product spaces.

Using a standard technique, one can easily infer from (\ref{ineqHHelem}) that
the Hornich-Hlawka inequality (\ref{ineq_HH}) also works for all Lebesgue
spaces $L^{1}(\mu)$, and so also for all spaces that can be embedded linearly and
isometrically into an $L^{1}(\mu)$. The latter comment includes all
Lebesgue spaces $L^{p}(\mu)$ with $p\in\lbrack1,2])$---see Lindenstrauss and
Pe\l czy\'{n}ski \cite{LP}.

In 1946, Popoviciu~\cite{Pop1946} proved that every continuous function
$f:\mathbb{R}_{+}\rightarrow\mathbb{R}$ that vanishes at the origin and admits
a nondecreasing derivative of second order on $(0,\infty)$, verifies the
Hornich-Hlawka functional inequality (\ref{HHfunct_ineq+}). One can easily put
Popoviciu's result in full generality by showing that actually all continuous
$3$-convex functions on $\mathbb{R}_{+}$ taking values in an ordered Banach
space verify inequality (\ref{HHfunct_ineq}). This fact and its analogue
in the case of continuous $n$-convex functions,%
\begin{multline*}
f\bigl(\sum\nolimits_{i=1}^{n}x_{i}\bigr)-\sum\nolimits_{1\leq i_{1}<\cdots<i_{n-1}\leq
n}f(x_{i_{1}}+\cdots+x_{i_{n-1}})\\
+\sum\nolimits_{1\leq i_{1}<\cdots<i_{n-2}\leq n}f(x_{i_{1}}+\cdots
+x_{i_{n-2}})-\cdots+(-1)^{n-1}\sum\nolimits_{i=1}^{n}f(x_{i})\geq f(0),
\end{multline*}
will be the subject of Section 3.

Close to the above inequality is the
characterization of the property of $n$-convexity via differences $\left(
\Delta_{h}f\right)  (x)=f(x+h)-f(x)$, rather than via divided differences as is
usual. See Theorem~\ref{thmH_P_BW_K}, which expresses the identity of the class
of continuous $n$-convex functions with the class of continuous functions
having positive differences of order $n$ in the sense that $\Delta_{x_{1}%
}\Delta_{x_{2}}\cdots\Delta_{x_{n}}f(x)\geq0.$ The connection with Popoviciu's
inequality is evident when $n$ is an odd integer because the condition
$\Delta_{x_{1}}\Delta_{x_{2}}\cdots\Delta_{x_{n}}f(x)\geq0,$ simply means the
introduction of a new variable in Popoviciu's inequality as follows:%
\begin{gather*}
\sum\nolimits_{i=1}^{n}f(x_{i}+x)-\sum\nolimits_{1\leq i<j\leq n}f(x_{i}%
+x_{j}+x)\\
+\sum\nolimits_{1\leq i<j<k\leq n}f(x_{i}+x_{j}+x_{k}+x)-\cdots\\
\hspace{1.5in}+(-1)^{n-1}f(x_{1}+\cdots+x_{n}+x)\geq f(x).
\end{gather*}

It is worth noticing that the functions $f:\mathbb{R}_{+}\rightarrow
\mathbb{R}_{+}$ that have positive differences of any order are precisely the
absolutely monotonic functions in the terminology of Bernstein~\cite{Bern1929}. Leaving the elegant framework of analysis on intervals one easily discovers
that $n$-convexity and the property of having positive differences of order
$n$ are different concepts. This idea is detailed at the end of Section 3.

Two important classes of functions that mix a string of properties of
$n$-convexity are those of completely monotone functions and of Bernstein
functions. See Section 2 for their definitions and some examples. Sendov and
Zitikis \cite{SZ2014} prove that these functions verify inequalities of the
form
\begin{equation}
\sum\nolimits_{i=1}^{n}f(x_{i})-\sum\nolimits_{1\leq i<j\leq n}f(x_{i}%
+x_{j})+\cdots+(-1)^{n-1}f\bigl(\sum\nolimits_{i=1}^{n}x_{i}\bigr)\geq0,\label{ineq_SZ}%
\end{equation}
for all $x_{1},\ldots,x_{n}$ in $\mathbb{R}_{+}$ and $n\geq1$. Their proof
combines the classical integral representation theorems (respectively the
Bernstein theorem and the L\'{e}vy--Khintchine representation theorem) with
some probabilistic considerations. In Section 4 we extend this result as a
double inequality that holds for completely monotone functions defined on
cones. Combining this result with \cite[Theorem 1.3~(a)]{SS2014}, 
we then show that the function $f(X)=(\det X)^{-\rho}$ (defined on
$N\times N$-real symmetric positive definite matrices) also
verifies the whole string of inequalities (\ref{ineq_SZ}) if $\rho\in\left\{
0,1/2,1,3/2,...\right\}  \cup\lbrack\left(  N-1\right)  /2,\infty)$.

Section~\ref{sec:cones} considers functions defined on cones and having positive
differences of a certain order $n>0.$ A surprising result is Theorem
\ref{thm_n-mon of det}, which shows that the function $\det$ has positive
differences of any order (though it is not completely monotonic). Probably the
same happens for other immanants function (like the permanents), but we were
able to prove only the positivity of differences of order 3. [TODO TODO].


For the reader's convenience, background on higher order convexity and the theory of ordered Banach spaces is summarized in Section~\ref{sec:back}.

\section{Preliminaries}
\label{sec:back}
The study of higher order convexity was initiated by Hopf \cite{Hopf} and
Popoviciu \cite{Pop34,Pop1944}, who defined it in terms of divided
differences of a function. Assuming $f$ a real-valued function defined on a
real interval $I$, the divided differences of order $0,1,\ldots,n$ associated
to a family $x_{0},x_{1},\ldots,x_{n}$ of $n+1$ distinct points are
respectively defined by the formulas%
\begin{align*}
\lbrack x_{0};f]  &  =f(x_{0})\\
\lbrack x_{0},x_{1};f]  &  =\frac{f(x_{1})-f(x_{0})}{x_{1}-x_{0}}\\
&  ...\\
\lbrack x_{0},x_{1},...,x_{n};f]  &  =\frac{[x_{1},x_{2},...,x_{n}%
;f]-[x_{0},x_{1},...,x_{n-1};f]}{x_{n}-x_{0}}\\
&  =%
{\displaystyle\sum\nolimits_{j=0}^{n}}
\frac{f(x_{j})}{\prod\nolimits_{k\neq j}\left(  x_{j}-x_{k}\right)  }.
\end{align*}
Notice that all these divided differences are invariant to permutations
of the points $x_{0},x_{1},...,x_{n}.$ As a consequence, we may always assume that
$x_{0}<x_{1}<\cdots<x_{n}.$

A function $f$ is called $n$-\emph{convex }(respectively\emph{ }%
$n$-\emph{concave}) if all divided differences $[x_{0},x_{1},\ldots,x_{n};f]$
are nonnegative (respectively nonpositive). In particular, $0$-convex
functions are precisely the nonnegative functions, $1$-convex
functions the nondecreasing ones, while $2$-convex
functions are simply the usual convex functions.

If $f$ is $n$ times differentiable, then a repeated application of Lagrange's
mean value theorem yields the existence of a point $\xi\in\left(  \min
_{k}x_{k},\max_{k}x_{k}\right)  $ such that%
\[
\lbrack x_{0},x_{1},...,x_{n};f]=\frac{f^{(n)}(\xi)}{n!}.
\]

As a consequence, one obtains the following practical criterion of $n$-convexity.

\begin{lemma}
\label{lemH&P}Every continuous function $f$ defined on an interval $I$ which
is $n$ times differentiable on the interior of $I$ is $n$-convex provided that
$f^{(n)}\geq0$.
\end{lemma}

A big source of convex functions of higher order is provided by the Bernstein
functions and the completely monotone functions. Recall that a function
$f:(0,\infty)\rightarrow\mathbb{R}_{+}$ is a \emph{Bernstein function} if it
is infinitely differentiable and verifies the condition
\[
(-1)^{n+1}f^{(n)}(x)\geq0\text{\quad for all }x>0\text{ and }n\geq1\text{;}%
\]
while, the function $f$ is \emph{completely monotone} if instead
\[
(-1)^{n}f^{(n)}(x)\geq0\text{\quad for all }x>0\text{ and }n\geq0.
\]

By definition, a function $f:[0,\infty)\rightarrow\mathbb{R}_{+}$ is a
\emph{Bernstein function} (respectively a \emph{completely monotone }function)
if it is continuous and its restriction of to $(0,\infty)$ has the respective property.

Every Bernstein function is $(2n+1)$-convex and every completely monotone
function is $2n$-convex for every $n\geq0.$

If $f:[0,\infty)\rightarrow\mathbb{R}_{+}$ is a Bernstein function then so is
$f-f(0);$ if $f$ is a completely monotone function then $f(0)-f$ is a
Bernstein function. Some simple examples of Bernstein functions are%
\begin{gather*}
x/(x+1),~1-e^{-\alpha x}\text{ (for }\alpha>0),\text{ }\\
\log(1+x),\text{ }(x-1)/\log x\text{ and }x^{\alpha}\text{ }(\text{for
}0<\alpha\leq1).
\end{gather*}
A nice account of the two aforementioned classes of functions is offered by
the authoritative monograph of Schilling, Song and Vondra\v{c}ek \cite{SSV}.

Besides the five examples mentioned above some other examples \ of $3$-convex
functions on $\mathbb{R}_{+}$ are $x^{\alpha}$ (for $\alpha\in(0,1]\cup
\lbrack2,\infty)),$ $-x^{2}+\sqrt{x},$ $-x\log x,$ $\sinh,$ $\cosh$,
$-\log\left(  \Gamma(x)\right)  $ etc.

The function $1-\left(  x-3\right)  +\frac{\left(  x-3\right)  ^{3}}{6}$ is
continuous and $3$-convex on $\mathbb{R}_{+}$ but not $n$-convex for any
$n\in\left\{  0,1,2\right\}  .$

The polynomials with positive coefficients and the exponential are $n$-convex
for every $n\geq0.$

All polynomials of degree less than or equal to 2 are both $3$-convex and $3$-concave.

The following approximation theorem due to Popoviciu~\cite{Pop35} (see
also \cite[Theorem 1.3.1 $(i)$, pg.~20]{Gal2008}) allows us to reduce reasoning
with $n$-convex functions to the case where they are also differentiable.

\begin{theorem}[Popoviciu's approximation theorem]
  \label{thm_Pop35} If a continuous function $f:[0,1]\rightarrow\mathbb{R}$ is $k$-convex, then so are the
Bernstein polynomials associated to it,
\[
B_{n}(f)(x)=\sum_{i=0}^{n}\binom{n}{i}x^{i}(1-x)^{n-i}f\left(  \frac{i}%
{n}\right)  .
\]
Moreover, by the well-known property of simultaneous uniform approximation of
a function and its derivatives by Bernstein polynomials and their
derivatives, it follows that $B_{n}(f)$ and any derivative (of any order) of
it, converge uniformly to $f$ and to its derivatives, correspondingly.
\end{theorem}

Using a change of variable, one can easily see that the approximation theorem
extends to functions defined on compact intervals $[a,b]$ with $a<b.$

\begin{lemma}
\label{lem_comp}$(i)$ The composition of two continuous functions that are
increasing, concave or $3$-convex is a function of the same nature.

$(ii)$ If $f:\mathbb{R}_{+}\rightarrow\mathbb{R}_{+}$ is a continuous
$3$-convex function which is also nondecreasing and concave, then the same
properties hold for $f^{\alpha}$ if $\alpha\in(0,1].$
\end{lemma}

\begin{proof}
According to Theorem~\ref{thm_Pop35}, we may reduce the proof to the case
where the involved functions are also of class $C^{3}$. In this case the proof
can be completed by computing the sign of the derivatives of order 1, 2 and 3.
\end{proof}

The $n$-convex functions taking values in an ordered Banach space can be
introduced in the same manner as real-valued $n$-convex functions by
using divided differences. We recall useful definitions below.

Recall that an \emph{ordered Banach space} is any Banach space $E$ endowed
with the ordering $\leq$ associated to a closed convex cone $E_{+}$ via the
formula
\[
x\leq y\ \text{if and only if }y-x\in E_{+},~
\]
such that%
\[
E=E_{+}-E_{+},\text{\quad}\left(  -E_{+}\right)  \cap E_{+}=\left\{
0\right\}  ,\text{ }%
\]
and%
\[
0\leq x\leq y\ \text{in}~E~\text{implies }\left\Vert x\right\Vert
\leq\left\Vert y\right\Vert .
\]
The basic facts concerning the theory of ordered Banach spaces are made
available by the book of Schaefer and Wolff \cite{SW}. See \cite{NO2020} for a
short overview centered on two important particular cases: $\mathbb{R}^{n},$
the $n$-dimensional Euclidean space endowed with the coordinate-wise ordering,
and $\operatorname*{Sym}(n,\mathbb{R)}$ the ordered Banach space of all
$n\times n$ symmetric matrices with real coefficients endowed with
the operator norm%
\[
\left\Vert A\right\Vert =\sup_{\left\Vert x\right\Vert \leq1}\left\vert
\langle Ax,x\rangle\right\vert,
\]
and the L\"{o}wner (partial) ordering,%
\[
A\leq B\text{ if and only if }\langle A\mbf{x},\mbf{x}\rangle\leq\langle
B\mbf{x},\mbf{x}\rangle\text{ for all }\mbf{x}\in\mathbb{R}^{n}.
\]

Here the operator norm can be replaced by any Schatten norm, in particular
with the \emph{Frobenius norm},
\[
\left\Vert A\right\Vert _{F}=\bigl(  \sum\nolimits_{i=1}^{N}\sum\nolimits_{j=1}^{N}a_{ij}%
^{2}\bigr)^{1/2}.
\]
The Frobenius norm is associated to the trace inner product%
\[
\langle A,B\rangle=\operatorname*{trace}(AB).
\]

The positive cone of $\mathbb{R}^{n}$ is the first orthant $\mathbb{R}_{+}%
^{n},$ while the positive cone of $\operatorname*{Sym}(n,\mathbb{R)}$ is the
set $\operatorname*{Sym}^{+}(n,\mathbb{R)}$ consisting of all positive
semi-definite matrices. We denote by $\mathbb{R}_{++}^{n}$ and
$\operatorname*{Sym}^{++}(n,\mathbb{R)}$ respectively the interior of
$\mathbb{R}_{+}^{n}$ and $\operatorname*{Sym}^{+}(n,\mathbb{R).}$

\begin{remark}
\label{rem_order}Much of the study of vector-valued convex functions can be
reduced to that of real-valued functions.\ Indeed, in any ordered Banach space
$E,$ any inequality of the form$\ u\leq v$ is equivalent to $x^{\ast}(u)\leq
x^{\ast}(v)$ for all $x^{\ast}\in E_{+}^{\ast}$.

As a consequence, a function $f:I\rightarrow E$ is respectively nondecreasing,
convex or $n$-convex if and only if $x^{\ast}\circ f$ has this property
whenever $x^{\ast}\in E^{\ast}$ is a positive functional. For $E=\mathbb{R}%
^{n},$ this inequality reduces to the components of $f$.
\end{remark}

Combining Remark \ref{rem_order} with Lemma \ref{lemH&P} one obtains the
following practical test of $3$-convexity for vector-valued differentiable functions:

\begin{theorem}
\label{lem_vector_3conv}Suppose that $f$ is a continuous function defined on
an interval $I$ and taking values in an ordered Banach space $E$. If $f$ is
three times differentiable on the interior of $I$ and $f^{\prime\prime\prime
}\geq0,$ then $f$ is a $3$-convex function.
\end{theorem}

An example illustrating Theorem \ref{lem_vector_3conv} is provided by the
function%
\[
f:\mathbb{R}_{+}\rightarrow\operatorname*{Sym}(n,\mathbb{R)},\text{\quad
}f(t)=-e^{-tA},
\]
associated to a positive semi-definite matrix $A\in\operatorname*{Sym}%
(n,\mathbb{R)}.$ This function is of class $C^{\infty}$ and its first three
derivatives are given by the formulas%
\[
f^{\prime}(t)=Ae^{-tA},\text{\quad}f^{\prime\prime}(t)=-A^{2}e^{-tA}%
,\text{\quad}f^{\prime\prime\prime}(t)=A^{3}e^{-tA}.
\]
Thus $f$ is nondecreasing, concave and 3-convex (according to the
ordering of $\operatorname*{Sym}(n,\mathbb{R))}.$ The matrix $A^{3}e^{-tA}$
is positive semidefinite since the product of commuting positive
semi-definite matrices is also positive semidefinite.

\section{The functional inequality of Popoviciu}

Popoviciu \cite{Pop1946} published in 1946 a short note on a functional
inequality that we restate here in a slightly more general form.

\begin{theorem}
\label{thm_Pop46}Suppose that $E$ is an ordered Banach space and
$f:[0,A]\rightarrow E$ is a continuous $n$-convex function $(n\geq1).$ Then%
\begin{multline*}
f(\sum\nolimits_{i=1}^{n}x_{i})-\sum\nolimits_{1\leq i_{1}<\cdots<i_{n-1}\leq
n}f(x_{i_{1}}+\cdots+x_{i_{n-1}})\\
+\sum\nolimits_{1\leq i_{1}<\cdots<i_{n-2}\leq n}f(x_{i_{1}}+\cdots
+x_{i_{n-2}})-\cdots+(-1)^{n-1}\sum\nolimits_{i=1}^{n}f(x_{i})\geq
(-1)^{n-1}f(0),
\end{multline*}
for all $x_{1},x_{2},...,x_{n}\geq0$ with $\sum\nolimits_{i=1}^{n}x_{i}\leq
A.$
\end{theorem}

Notice that this inequality can be reformulated as%
\[%
{\displaystyle\sum\limits_{\varepsilon_{1},...,\varepsilon_{n}\in\left\{
0,1\right\}  }}
\left(  -1\right)  ^{n-\left(  \varepsilon_{1}+\cdots+\varepsilon_{n}\right)
}f\left(  \varepsilon_{1}x_{1}+\cdots+\varepsilon_{n}x_{n}\right)  \geq0.
\]

Popoviciu supplied the details only in the case $n=3$, for functions $f$ that
vanish at the origin and admit a nondecreasing second order derivative.

The proof of Theorem \ref{thm_Pop46} is by induction, starting with the following instance of 
Hardy-Littlewood-P\'{o}lya's majorization inequality (see \cite[Theorem 4.1.3, pg.~186]{NP2018}):

\begin{lemma}
\label{lem_maj}\emph{If }$g:\left[  a,b\right]  \rightarrow%
\mathbb{R}
$ \emph{is a continuous convex function and }$c$ and $d$ are two points in
$\left[  a,b\right]  $ such that $a+b=c+d,$ then%
\[
g(c)+g(d)\leq g(a)+g(b).
\]

\end{lemma}

\begin{proof}[\textbf{Proof of Theorem~\ref{thm_Pop46}}] The case $n=1$ is trivial since 1-convexity is equivalent to the fact that $f$ is nondecreasing. For $n=2$ the inequality 
under attention reads as%
\[
f(x_{1})+f(x_{2})\leq f(x_{1}+x_{2})+f(0),
\]
which follows from Lemma \ref{lem_maj}. Suppose thus that the statement of Theorem \ref{thm_Pop46} holds for all continuous $n$-convex functions and all $x_{1},x_{2},...,x_{n}\geq0$ with
$\sum\nolimits_{i=1}^{n}x_{i}\leq A.$

Let $f$ be a continuous $\left(  n+1\right)  $-convex function. According to
Remark \ref{rem_order} we may assume that $f$ is real-valued, while
Popoviciu's approximation theorem (Theorem~\ref{thm_Pop35}) allows us to
restrict ourselves functions of class $C^{1}.$ Then
$f^{\prime}$ is continuous and $n$-convex and the same is true for the
function $\varphi(x)=f^{\prime}(x_{1}+x)$. According to the induction
hypothesis, if $x_{1},...,x_{n},x_{n+1}\geq0$ and $x_{1}+\cdots+x_{n+1}\leq
A,$ we have%
\[
f^{\prime}(\sum\nolimits_{i=1}^{n+1}x_{i})-\sum\nolimits_{2\leq i_{1}%
<\cdots<i_{n-1}\leq n+1}f^{\prime}(x_{1}+x_{i_{1}}+\cdots+x_{i_{n-1}})
\]

\[
+\sum\nolimits_{2\leq i_{1}<\cdots<i_{n-2}\leq n+1}f^{\prime}(x_{1}+x_{i_{1}%
}+\cdots+x_{i_{n-2}})
\]

\[
-\cdots+(-1)^{n-1}\sum\nolimits_{2\leq j\leq n+1}f^{\prime}(x_{1}+x_{j}%
)\geq(-1)^{n-1}f^{\prime}(x_{1}).
\]

Similar inequalities occur by permuting the variables.

Consider $x_{2},x_{3},\ldots,x_{n+1}~$fixed in $[0,A]$ and $x_{1}\geq0$
variable such that $x_{1}+\cdots+x_{n+1}\leq A.$ The function $F$ defined by
the formula%
\begin{multline*}
F\left(  x_{1}\right)  =f(\sum\nolimits_{i=1}^{n+1}x_{i})-\underset{1\leq
i_{1}<\cdots<i_{n}\leq n+1}{\sum}f(x_{i_{1}}+\cdots+x_{i_{n}})+\cdots+\\
+(-1)^{n-1}\underset{1\leq i<j\leq n+1}{\sum}f(x_{i}+x_{j})+(-1)^{n}%
\sum\nolimits_{i=1}^{n+1}f(x_{i})+\left(  -1\right)  ^{n+1}f\left(  0\right)
\end{multline*}
is differentiable and, according to the induction hypothesis,%
\begin{multline*}
F^{\prime}\left(  x_{1}\right)  =f^{\prime}\left(  \sum\nolimits_{i=1}%
^{n+1}x_{i}\right)  -\underset{2\leq i_{1}<\cdots<i_{n+1}\leq n+1}{\sum
}f^{\prime}(x_{1}+x_{i_{1}}+\cdots+x_{i_{n-1}})+\\
+(-1)^{n-1}\underset{2\leq j\leq n+1}{\sum}f^{\prime}(x_{1}+x_{j}%
)+(-1)^{n}f^{\prime}(x_{1})\geq0.
\end{multline*}
Therefore $F$ is a nondecreasing function, whence $F\left(  x_{1}\right)  \geq
F\left(  0\right)  =0$. In conclusion $F$ is a nonnegative function and the
proof is done.
\end{proof}

It is worth noticing that Popoviciu's inequality can be turned into a
characterization of $n$-convexity using difference operators.

The \emph{difference operators} $\Delta_{h}$ $($of step size $h\geq0)$ can be
introduced in a large category of situations including the case of functions
defined on $n$-dimensional intervals or on convex cones, ordered abelian
semigroups, etc. They associate to each such function $f$, the function
$\Delta_{h}f$ defined by%
\[
\left(  \Delta_{h}f\right)  (x)=f(x+h)-f(x),
\]
for all $x$ and $h$ such that the right-hand side formula makes sense.

Clearly, difference operators are linear and commute with each other,%
\[
\Delta_{h_{1}}\Delta_{h_{2}}=\Delta_{h_{2}}\Delta_{h_{1}}.
\]

They also verify the following property of invariance under translation:%
\[
\Delta_{h}\left(  f\circ T_{a}\right)  =\left(  \Delta_{h}f\right)  \circ
T_{a},
\]
where $T_{a}$ is the translation defined by the formula $T_{a}(x)=x+a.$

\begin{lemma}
\label{lem_iter}If $n$ is a positive integer, then the following formula
holds:
\[
\Delta_{h_{1}}\Delta_{h_{2}}\cdots\Delta_{h_{n}}f(x)=%
{\displaystyle\sum\limits_{\varepsilon_{1},...,\varepsilon_{n}\in\left\{
0,1\right\}  }}
\left(  -1\right)  ^{n-\left(  \varepsilon_{1}+\cdots+\varepsilon_{n}\right)
}f\left(  x+\varepsilon_{1}h_{1}+\cdots+\varepsilon_{n}h_{n}\right)  .
\]

\end{lemma}

The proof is immediate, by mathematical induction.

The property of convexity of a continuous function $f$ defined on an interval
$I$ can be characterized via the difference operators as follows:

\begin{theorem}
\label{thm_2conv=2mon}A continuous function $f:I\rightarrow\mathbb{R}$ is
convex if and only if the following inequality holds,%
\begin{equation}
\Delta_{a}\Delta_{b}f(x)=f(a+b+x)-f(a+x)-f(b+x)+f(x)\geq0 \label{ineq_2-mon}%
\end{equation}
at all interior points $x\in I$ and all $a,b\geq0$ for which $x+a+b\in I.$
\end{theorem}

In other words, for continuous functions defined on intervals, convexity is equivalent to the property of having positive differences of second order.

\begin{proof}
The fact that convexity implies the inequality $\Delta_{a}\Delta_{b}f\geq0$ is
a consequence of the Hardy-Littlewood-P\'{o}lya inequality of majorization.
See Lemma \ref{lem_maj}. On the other hand, for $u<v$ arbitrarily fixed in
$I,$ choosing $a=b=\left(  v-u\right)  /2$ and $x=u,$ we infer from
(\ref{ineq_2-mon}) that
\[
\frac{f(u)+f(v)}{2}\geq f\left(  \frac{u+v}{2}\right)  ,
\]
which is equivalent to convexity since $f$ was assumed to be continuous. See
\cite[Theorem 1.1.8, pg.~5]{NP2018}.
\end{proof}

The following result extends Theorem \ref{thm_2conv=2mon} to the case of
higher-order convexity and originates from an old paper of Boas and Widder
\cite{BW1940}.

\begin{theorem}
\label{thmH_P_BW_K}Suppose that $f:[0,A]\rightarrow\mathbb{R}$ is a continuous
function and $n\geq1$ is an integer. Then the following conditions are equivalent:

$(i)$ $f$ is $n$-convex$;$

$(ii)$ $f$ has positive differences of order $n$ in the sense that
\[
\Delta_{x_{1}}\Delta_{x_{2}}\cdots\Delta_{x_{n}}f(t)\geq0
\]
for all points $t,x_{1},x_{2},...,x_{n}\geq0$ such that $t+x_{1}+\cdots
+x_{n}\leq A.$
\end{theorem}

The same works if the interval $[0,A]$ is replaced by $\mathbb{R}_{+}$ and
$\mathbb{R}.$

\begin{proof}
The implication $(i)\Longrightarrow(ii)$ follows from Theorem \ref{thm_Pop46},
when applied to the $n$-convex function $g(x)=f(x+t)-f(t).$ An alternative
proof is made available by the paper of Boas and Widder \cite{BW1940}.

The converse implication is immediate and is the objective of \cite[Theorem
15.3.1, pg.~430]{Kuc2009} in Kuczma's book. A very short proof of
the implication $(ii)\Longrightarrow(i)$ when $n=3$ can be found in \cite[Proposition 1]{Ben2010}.
\end{proof}

\begin{corollary}
\label{cor1}Suppose that $E$ is an ordered Banach space and
$f:[0,A]\rightarrow E$ is a continuous function. Then $f$ is $n$-convex if and
only if it verifies the inequality%
\[
\Delta_{x_{1}}\Delta_{x_{2}}\cdots\Delta_{x_{n}}f(x)=%
{\displaystyle\sum\limits_{\varepsilon_{1},...,\varepsilon_{n}\in\left\{
0,1\right\}  }}
\left(  -1\right)  ^{n-\left(  \varepsilon_{1}+\cdots+\varepsilon_{n}\right)
}f\left(  x+\varepsilon_{1}x_{1}+\cdots+\varepsilon_{n}x_{n}\right)  \geq0
\]
for all points $x,x_{1},\ldots,x_{n}\in\lbrack0,A]$ such that $x+x_{_{1}%
}+\cdots+x_{n}\leq A$.

In particular, a continuous function $f:[0,A]\rightarrow E$ is $3$-convex if
and only if it verifies the inequality%
\begin{multline*}
f\left(  x+t\right)  +f\left(  y+t\right)  +f\left(  z+t\right)  +f\left(
x+y+z+t\right) \\
\geq f\left(  x+y+t\right)  +f\left(  y+z+t\right)  +f\left(  z+x+t\right)
+f(t)
\end{multline*}
for all points $x,y,z,t\in\lbrack0,A]$ such that $x+y+z+t\leq A$.
\end{corollary}

Both Theorem \ref{thm_Pop46} and Theorem \ref{thmH_P_BW_K} can be applied
successfully to derive a number of useful inequalities satisfied by the
Bernstein functions and the completely monotonic functions on $[0,\infty).$ We
will come back to this matter in the next section.

In higher dimensions, the equivalence between usual convexity and the property
of having positive differences of second order is no anymore valid. Some
simple examples are indicated in what follows.

\begin{example}
Consider the case of the infinitely differentiable function
\[
f(x,y)=-2\left(  xy\right)  ^{1/2},\quad x,y\in(0,\infty).
\]
This function is convex, its Hessian being the positive semidefinite matrix
\[
H=\frac{1}{2}\left(
\begin{array}
[c]{cc}%
x^{-3/2}y^{1/2} & -x^{-1/2}y^{-1/2}\\
-x^{-1/2}y^{-1/2} & x^{1/2}y^{-3/2}%
\end{array}
\right)  .
\]
However, $\Phi$ fails the inequality%
\[
\Delta_{A}\Delta_{B}f(X)\geq0\quad\text{for }A,B,X\in(0,\infty)\times
(0,\infty);
\]
for example, choose $A=(1,2),$ $B=(2,1)$ and $X$ near the origin.
\end{example}

\begin{example}
The function
\[
M:\mathbb{R}_{+}^{2}\rightarrow\mathbb{R},\quad M\left(  x,y\right)
=\min\left\{  x,y\right\}
\]
is continuous and concave. Besides it has positive differences of second order
as%
\begin{multline*}
\min\left\{  x+s+u,y+t+v\right\}  -\min\left\{  x+s,y+t\right\}  \\
-\min\left\{  x+u,y+v\right\}  +\min\left\{  x,y\right\}  \geq0,
\end{multline*}
for all $s,t,u,v>0.$ The function $M$ proves useful in statistics as the
Fr\'{e}chet-Hoeffding upper bound for joint distribution functions of random
variables. See Nelsen \cite{Nelsen}.
\end{example}

Popoviciu \cite{Pop34,Pop1944} introduced the concept of higher
order convexity for functions of several variables using multiple divided
differences. To gain some insight, let us consider the case
of a function $f=f(x,y)$ defined on a product $I\times J$ of intervals, and let
$x_{0},x_{1},\ldots,x_{m}$ be distinct points in $I$, and $y_{0},y_{1}%
,\ldots,y_{n}$ be distinct points in $J.$ The \emph{divided double
differences} are defined via the formula%
\begin{align}
\left[
\begin{array}
[c]{cccc}%
x_{0}, & x_{1}, & \ldots & ,~x_{m}\\
y_{0}, & y_{1}, & \ldots & ,\ y_{n}%
\end{array}
;f\right]   &  =[x_{0},x_{2},\ldots,x_{m};[y_{0},y_{1},\ldots,y_{n}%
;f((x,\cdot)]]\label{defdivdiff}\\
&  =[y_{0},y_{1},\ldots,y_{n};[x_{0},x_{1},\ldots,x_{m};f((\cdot
,y)]]\text{.}\nonumber
\end{align}
Notice that this formula is invariant under the permutation of variables
$x_{k}$ (and also under the permutation of the variables $y_{k}).$

Drawing a parallel to the one dimensional case, Popoviciu \cite[pg.~78]{Pop34} calls a function $f:I\times J\rightarrow\mathbb{R}$
\emph{convex} \emph{of order} $(m,n)$ if the divided differences%
\[
\left[
\begin{array}
[c]{cccc}%
x_{0}, & x_{1}, & \ldots & ,~x_{m}\\
y_{0}, & y_{1}, & \ldots & ,~y_{n}%
\end{array}
;f\right]
\]
are nonnegative for all distinct points $x_{0},x_{1},...,x_{m}\in I$ and
$y_{0},y_{1},...,y_{n}\in J$.

Needless to say, the study of this concept of convexity implies a formidable
formalism, so little progress was made since the times of Popoviciu. The only one
recent contribution\ is \cite{GN2019} that studies the cases 
$m=n=1$ and $m=n=2.$ 

\section{The case of completely monotone functions on cones}

The theory of completely monotone functions can be easily extended to the
context of several variables using convex analysis. In what follows $V$
denotes a finite-dimensional real vector space and $\mathcal{C}$ an open convex cone in
$V$ with closure \ $\overline{\mathcal{C}}$. Its dual cone is $\mathcal{C}^{\ast}=\{y\in E^{\ast
}:\langle y,x\rangle\geq0$ for all $x\in \mathcal{C}\}$. The points in $C^{\ast}$ are
linear functionals that are nonnegative on $\overline{\mathcal{C}}$.

\begin{definition}
A function $f:\mathcal{C}\rightarrow$ $\mathbb{R}_{+}$ is called completely monotone if
$f$ is $\mathcal{C}^{\infty}$ on $\mathcal{C}$ and, for all integers $k\geq1$
and all vectors $v_{1},...,$ $v_{k}$ $\in \mathcal{C}$, we have%
\begin{equation}
\left(  -1\right)  ^{k}D_{v_{1}}\cdots D_{v_{k}}f(x)\geq0\text{\quad for all
}x\in \mathcal{C}. \label{def_cm}%
\end{equation}
Here $D_{v}$ denotes the directional derivative along the vector $v$.

A function $f:\overline{\mathcal{C}}\rightarrow$ $\mathbb{R}_{+}$ is called completely
monotone if it is the continuous extension of a completely monotone function
on $\mathcal{C}.$
\end{definition}

When $\mathcal{C}=(0,\infty)^{n},$ the condition (\ref{def_cm}) means that%
\[
\left(  -1\right)  ^{k}\frac{\partial^{k}f}{\partial x_{i_{1}}\partial
x_{i_{2}}\cdots\partial x_{i_{k}}}(x)\geq0
\]
for all $x\in\left(  0,\infty\right)  ^{n}$ and all sets of indices $1\leq
i_{1}\leq i_{2}\leq\cdots\leq i_{k}\leq n$ of arbitrary length $k.$

As in the case of completely monotone functions of one real variable, these
functions can be obtained as Laplace transforms of Borel measures on the dual cone.

\begin{theorem}
$($Bernstein-Hausdorff-Widder-Choquet theorem$)$.\label{thm_BHWCh} Let $f$ be
a nonnegative continuous function on the open convex cone $\mathcal{C}$. Then $f$ is
completely monotone if and only if it is the Laplace transform of a unique
Borel measure $\mu$ supported on the dual cone $\mathcal{C}^{\ast}$, that is,%
\[
f(x)=\int_{\mathcal{C}^{\ast}}e^{-\langle y,x\rangle}\mathrm{d}\mu(y)\text{\quad for all
}x\in \mathcal{C}.
\]

When $f$ admits a continuous extension to $\overline{\mathcal{C}},$ the last equality
works for all $x\in\overline{\mathcal{C}}.$
\end{theorem}

For details, see Choquet \cite{Ch1969}.

\begin{remark}
\label{rem_cm on cones}Finding the positive Borel measure $\mu$ that makes the
formula of Theorem \emph{\ref{thm_BHWCh}} working represents a practical way
for checking the complete monotonicity of $f$. So is the case of Riesz
kernels: If $\alpha_{1},\alpha_{2},...,\alpha_{N}>0,$ then
\[
x_{1}^{-\alpha_{1}}x_{2}^{-\alpha_{2}}\cdots x_{N}^{-\alpha_{N}}%
=\int_{\mathbb{R}_{++}^{N}}e^{-\langle\mbf{y},\mbf{x}\rangle}\frac
{x_{1}^{-\alpha_{1}}x_{2}^{-\alpha_{2}}\cdots x_{N}^{-\alpha_{N}}}%
{\Gamma(\alpha_{1})\Gamma(\alpha_{2})\cdots\Gamma(\alpha_{N})}\mathrm{d}%
\mbf{y}%
\]
for all $\mbf{x}\in\mathbb{R}_{++}^{N}.$ See \cite[Proposition~2.7]{KMS2019}. A more subtle case is that of inverse powers of the determinant%
\[
f\left(  X)=(\det X\right)  ^{-\rho},\quad X\in\operatorname*{Sym}%
\nolimits^{++}(N,\mathbb{R}),
\]
for which Scott and Sokal \cite{SS2014} have shown that is completely monotone
if and only if $\rho\in\left\{  0,1/2,1,3/2,...,\left(  N-1\right)
/2\right\}  \cup(\left(  N-1\right)  /2,\infty)$. See also \cite[Theorem~4.1]{KMS2019}. It is worth noticing that Siegel established in 1929 the
formula%
\[
\left(  \det A\right)  ^{-\rho}=\int_{\operatorname*{Sym}^{++}(N,\mathbb{R)}%
}e^{-\operatorname*{trace}AX}\frac{\left(  \det X\right)  ^{\rho}~\mathrm{d}%
X}{\pi^{n(n-1)/4}\Gamma(\rho)\Gamma\left(  \rho-1/2\right)  \cdots
\Gamma\left(  \rho-\left(  n-1)/2\right)  \right)  }%
\]
for all $A\in\operatorname*{Sym}^{++}(N,\mathbb{R)}$ and $\rho\geq\left(
N+1\right)  /2.$ See \emph{\cite[Hilfssatz 37, pg.~585]{Siegel}.}
\end{remark}

We next extend (and improve) a result due to Sendov and Zitikis; see
\cite[Theorem 4.1, pg.~76]{SZ2014}.

\begin{theorem}
\label{thm_SZ on cones}Every completely monotone function $f:\mathcal{C}\rightarrow
\mathbb{R}_{+}$ satisfies%
\begin{equation}
\sum\nolimits_{i=1}^{n}f(x_{i})-\sum\nolimits_{1\leq i<j\leq n}f(x_{i}%
+x_{j})+\cdots+(-1)^{n-1}f(\sum\nolimits_{i=1}^{n}x_{i})\geq0, \label{SZ1}%
\end{equation}
for every $x_{1},\ldots,x_{n}\in \mathcal{C}$ and $n\geq1.$ When $f$ admits a continuous
extension to $\overline{\mathcal{C}}$, then then $f$ satisfies the double inequality%
\begin{equation}
f(0)\geq\sum\nolimits_{i=1}^{n}f(x_{i})-\sum\nolimits_{1\leq i<j\leq n}%
f(x_{i}+x_{j})+\cdots+(-1)^{n-1}f(\sum\nolimits_{i=1}^{n}x_{i})\geq0,
\label{SZ2}%
\end{equation}
for every $x_{1},\ldots,x_{n}\in\overline{\mathcal{C}}$ and $n\geq1.$
\end{theorem}

For $n=3,$ the first conclusion of Theorem \ref{thm_SZ on cones} reads as
\[
\sum\nolimits_{i=1}^{3}f(x_{i})-\sum\nolimits_{1\leq i<j\leq3}f(x_{i}%
+x_{j})+f(\sum\nolimits_{i=1}^{3}x_{i})\geq0,
\]
which is nothing but a Hornich-Hlawka type inequality.

The proof of Theorem \ref{thm_SZ on cones} needs the following auxiliary result.

\begin{lemma}
\label{lem_BHWCh}We have%
\[
P=\sum\nolimits_{i=1}^{n}e^{-\alpha_{i}}-\sum\nolimits_{1\leq i<j\leq
n}e^{-(\alpha_{i}+\alpha_{j})}+\cdots+(-1)^{n+1}e^{-\sum\nolimits_{i=1}%
^{n}\alpha_{i}}\geq0
\]
and%
\begin{multline*}
Q=\sum\nolimits_{i=1}^{n}\left(  1-e^{-\alpha_{i}}\right)  -\sum
\nolimits_{1\leq i<j\leq n}\left(  1-e^{-(\alpha_{i}+\alpha_{j})}\right) \\
+\cdots+(-1)^{n+1}\left(  1-e^{-\sum\nolimits_{i=1}^{n}\alpha_{i}}\right)
\geq0,
\end{multline*}
whenever $\alpha_{1},\ldots,\alpha_{n}\in\mathbb{R}_{+}$ and $n\geq2.$
\end{lemma}

\begin{proof}
Indeed,%
\[
S=1-%
{\displaystyle\prod\nolimits_{i=1}^{n}}
\left(  1-e^{-\alpha_{i}}\right)  \text{ and }Q=\sum\nolimits_{i=1}%
^{n}(-1)^{k+1}\binom{n}{k}-S=1-S.
\]

\end{proof}

\begin{proof}[\textbf{Proof of Theorem~\ref{thm_SZ on cones}}] As per to Theorem \ref{thm_BHWCh},
$f$ admits the integral representation
\[
f(x)=\int_{\mathcal{C}^{\ast}}e^{-\langle y,x\rangle}\mathrm{d}\mu(y)\text{\quad for all
}x\in \mathcal{C},
\]
where $\mu$ is a Borel measure on $\mathcal{C}^{\ast}.$ Then, taking into account to the
first assertion of Lemma \ref{lem_BHWCh}, we have the inequality
\begin{multline*}
0\leq\int_{\mathcal{C}^{\ast}}\sum\nolimits_{i=1}^{n}e^{-\langle x_{i},y\rangle}%
-\sum\nolimits_{1\leq i<j\leq n}e^{-\langle x_{i}+x_{j},y\rangle}\\
\left.  +\cdots+(-1)^{n-1}e^{-\langle\sum\nolimits_{i=1}^{n}x_{i},y\rangle
}\right]  \mathrm{d}\mu(y)\\
=\sum\nolimits_{i=1}^{n}f(x_{i})-\sum\nolimits_{1\leq i<j\leq n}f(x_{i}%
+x_{j})+\cdots+(-1)^{n-1}f(\sum\nolimits_{i=1}^{n}x_{i}).
\end{multline*}
The case where $f$ is defined on $\overline{\mathcal{C}}$ can be settled in the same
manner, using both assertions of Lemma \ref{lem_BHWCh}.
\end{proof}

Combining Theorem \ref{thm_SZ on cones} with the aforementioned result of
Scott and Sokal (see Remark \ref{rem_cm on cones}) one obtains the following result:

\begin{corollary}
  \label{cor:det}
If $\rho\in\left\{  0,1/2,1,3/2,...\right\}  \cup\lbrack\left(  N-1\right)
/2,\infty)$, then
\begin{multline*}
\sum\nolimits_{i=1}^{n}\det\nolimits^{-\rho}(A_{i})-\sum\nolimits_{1\leq
i<j\leq n}\det\nolimits^{-\rho}(A_{i}+A_{j})\\
+\cdots+(-1)^{n-1}\det\nolimits^{-\rho}(\sum\nolimits_{i=1}^{n}A_{i})\geq0,
\end{multline*}
for every $A_{1},\ldots,A_{n}\in\operatorname*{Sym}\nolimits^{++}%
(N,\mathbb{R})$ and $n\geq1.$
\end{corollary}

Given Corollary~\ref{cor:det} one may wonder whether Theorem~\ref{thm_SZ on cones} also specializes to elementary symmetric polynomials. The situation here turns out to be more subtle, and a qualified answer follows from the discussion below. Recall that for any $m=0,1,...,N,$ the $m$-th elementary symmetric polynomial of
$\mbf{x}\in\mathbb{R}^{N}$ is defined by the formula%
\[
E_{m,N}(\mbf{x})=\sum\nolimits_{1\leq i_{1}<\cdots<i_{m}\leq N}x_{i_{1}%
}...x_{i_{m}}.
\]
Notice that $\det$ and the functions $E_{m,N}$ are hyperbolic polynomials in
the sense of G\aa{}rding. For details concerning this notion see
\cite[Section~4.6]{NP2018}. Kozhasov, Michalek and Sturmfels provided a
constructive proof for the fact that any elementary symmetric polynomial
$E_{m,n}$ admits a real exponent $\alpha^{\prime}>0$ such that $E_{m,n}%
^{-\alpha}$ is completely monotone on $\mathbb{R}_{++}^{N}$ for all
$\alpha\geq\alpha^{\prime}$---see \cite[Theorem 6.4]{KMS2019}. 

\begin{remark}
The restriction of $e^{-x}$ to $\mathbb{R}_{+}$ is a completely monotone
function and thus it verifies the inequality%
\[
e^{-x_{1}}+e^{-x_{2}}+e^{-x_{3}}+e^{-\left(  x_{1}+x_{2}+x_{3}\right)  }\geq
e^{-\left(  x_{1}+x_{2}\right)  }+e^{-\left(  x_{2}+x_{3}\right)
}+e^{-\left(  x_{3}+x_{1}\right)  }.
\]
However there $x_{1},x_{2},x_{3}>0$ such that%
\[
e^{-x_{1}}+e^{-x_{2}}+e^{-x_{3}}+e^{-\left(  x_{1}+x_{2}+x_{3}\right)
}\ngeqslant e^{-\left(  x_{1}+x_{2}\right)  }+e^{-\left(  x_{2}+x_{3}\right)
}+e^{-\left(  x_{3}+x_{1}\right)  }+e^{0}.
\]
This shows that the inequality
\[
\sum\nolimits_{i=1}^{3}f(x_{i})-\sum\nolimits_{1\leq i<j\leq3}f(x_{i}%
+x_{j})+f(\sum\nolimits_{i=1}^{3}x_{i})\geq f(0),
\]
does not characterize the $3$-convexity within the class of continuous functions.
\end{remark}

The fact that the Bernstein functions $f:\mathbb{R}_{+}\rightarrow
\mathbb{R}_{+}$ also verify the inequalities (\ref{SZ1}) follows from Lemma
\ref{lem_BHWCh} and the L\'{e}vy-Khintchine theorem (Theorem 3.2, p. 15 of
\cite{SSV}), which asserts that each such function admits the integral
representation
\[
f(x)=a+bx+\int_{0}^{\infty}(1-e^{-xt})\mathrm{d}\mu(t)
\]
for some constants $a,b\geq0$ and a positive measure $\mu$ on $[0,\infty)$
such that%
\[
\int_{0}^{\infty}\min\left\{  1,t\right\}  \mathrm{d}\mu(t)<\infty.
\]

As far as we know, no attempt was made to extend the theory of Bernstein
functions to the framework of functions defined on cones.

\section{Functions with positive differences on cones}
\label{sec:cones}
The difference operators $\Delta_{h}f:$ $x\rightarrow\Delta_{h}%
f(x)=f(x+h)-f(x)$ are well defined in the case of functions $f$ defined on a
convex cone $\mathcal{C}$. Such a function is said to be a \emph{function with positive
differences of order} $n$ ($n\geq1$) if
\[
\Delta_{h_{1}}\Delta_{h_{2}}\cdots\Delta_{h_{n}}f(x)\geq0\text{ \ for all
}x,h_{1},h_{2},...,h_{n}\in \mathcal{C}.
\]
For convenience, we say that $f$ has \emph{positive differences of order} $0$
if $f\geq0.$ 

In the literature, the concept of $n$-absolute monotonicity is used with the meaning that the function under attention has \emph{positive differences of
order} $k$ for all $k\in\{0,1,...,n\}.$

\begin{theorem}
If $f:\mathcal{C}\rightarrow F_{+}$ is a completely monotonic function, then
$f$ has positive differences of any even order $k=0,2,4,...~;$ under the same
hypothesis, $-f$ has positive differences of any odd order $k=1,3,5,...~.$
\end{theorem}

This follows easily from the characterization of monotonicity in terms of
G\^{a}teaux differentiability (as was established by Amann \cite{Amann1974},
Proposition 3.2, p. 184):

\begin{lemma}
\label{lemAmann}Suppose that $E$ and $F$ are two ordered Banach spaces,
$\mathcal{C}$ is a convex subset of $E$ with nonempty interior
$\operatorname{int}\mathcal{C}$ and $\Phi:\mathcal{C}\rightarrow F$ is a
function, continuous on $\mathcal{C}$ and G\^{a}teaux differentiable on
$\operatorname{int}\mathcal{C}.$ Then $\Phi$ is monotone nondecreasing on
$\mathcal{C}$ if and only if
\[
D\Phi(a)[v]=\frac{\Phi(a+tv)-\Phi(a)}{t}\geq0
\]
for all points $a\in\operatorname{int}\mathcal{C}$ and all vectors $v\in
E_{+}.$
\end{lemma}

\begin{example}
Every function of the form%
\[
\Phi(\mbf{x})=f\left(  \langle\mbf{x},\mbf{w}\rangle\right)
,\quad\mbf{x}\in\mathbb{R}_{+}^{N},
\]
associated to a continuous $n$-convex function $f:\mathbb{R}_{+}%
\mathbb{\rightarrow R}$ and a vector $\mbf{w}\in\mathbb{R}_{+}^{N}$ has
positive differences of order $n$.

The proof is straightforward, as evident from the case $n=3.$ Indeed, in this case,
for every $\mbf{x},\mbf{y},\mbf{z},\mbf{t}\in\mathbb{R}_{+}^{n}$
we have
\begin{multline*}
\Phi\left(  \mbf{x}+\mbf{t}\right)  +\Phi\left(  \mbf{y}%
+\mbf{t}\right)  +\Phi\left(  \mbf{z}+\mbf{t}\right)  +\Phi\left(
\mbf{x}+\mbf{y}+\mbf{z}+\mbf{t}\right) \\
-\Phi\left(  \mbf{x}+\mbf{y}+\mbf{t}\right)  -\Phi\left(
\mbf{y}+\mbf{z}+\mbf{t}\right)  -\Phi\left(  \mbf{z}%
+\mbf{x}+\mbf{t}\right)  -\Phi(\mbf{t})\\
=f\left(  \langle\mbf{x}+\mbf{t},\mbf{v}\rangle\right)  +f\left(
\langle\mbf{y}+\mbf{t},\mbf{v}\rangle\right)  +f\left(
\langle\mbf{z}+\mbf{t},\mbf{v}\rangle\right)  +f\left(
\langle\mbf{x}+\mbf{y}+\mbf{z}+\mbf{t},\mbf{v}\rangle\right) \\
-f\left(  \langle\mbf{x+y+t},\mbf{v}\rangle\right)  -f\left(
\langle\mbf{y+z+t},\mbf{v}\rangle\right)  -f\left(  \langle
\mbf{z+x+t},\mbf{v}\rangle\right)  -f(\langle\mbf{t},\mbf{v}%
\rangle)\\
=f\left(  \langle\mbf{x},\mbf{v}\rangle+\langle\mbf{t},\mbf{v}%
\rangle\right)  +f\left(  \langle\mbf{y},\mbf{v}\rangle+\langle
\mbf{t},\mbf{v}\rangle\right)  +f\left(  \langle\mbf{z}%
,\mbf{v}\rangle+\langle\mbf{t},\mbf{v}\rangle\right) \\
+f\left(  \langle\mbf{x}+\mbf{y}+\mbf{z},\mbf{v}\rangle
+\langle\mbf{t},\mbf{v}\rangle\right)  -f\left(  \langle\mbf{x+y}%
,\mbf{v}\rangle+\langle\mbf{t},\mbf{v}\rangle\right)  -f\left(
\langle\mbf{y+z},\mbf{v}\rangle+\mbf{t},\mbf{v}\rangle\right) \\
-f\left(  \langle\mbf{z+x},\mbf{v}\rangle+\langle\mbf{t}%
,\mbf{v}\rangle\right)  -f(\langle\mbf{t},\mbf{v}\rangle)\geq0.
\end{multline*}
A variant of this example is provided by the map%
\[
\Psi(A)=f\left(  \operatorname*{trace}(AW)\right)  ,\text{\quad}%
A\in\operatorname*{Sym}\nolimits^{+}(n,\mathbb{R)},
\]
associated to a continuous $n$-convex function $f:\mathbb{R}_{+}%
\mathbb{\rightarrow R}$ and to an operator $W\in\mathbb{R}_{+}^{N}.$
The ambient Hilbert space in this case is $\operatorname*{Sym}(n,\mathbb{R)},$
endowed with the Frobenius norm and L\"{o}wner ordering.
\end{example}

Linear algebra offers many examples of functions that have
positive differences of any order $n\geq0.$ So is the case of the determinant
function $\det$, restricted to $\operatorname*{Sym}\nolimits^{+}%
(N,\mathbb{R})$ (though $\det$ it\ is not completely monotonic). 

Clearly, $\det(X)\geq0$ and since $\det$ is monotonic on the semidefinite matrices, 
\[
\left(  \Delta_{A}\det\right)  (X)=\det(X+A)-\det(X)\geq0,
\]
for all $A,X\in\operatorname*{Sym}\nolimits^{+}(N,\mathbb{R}).~$Also simple is
the fact that%
\[
\left(  \Delta_{A}\left(  \Delta_{B}\det\right)  \right)  \left(  X\right)
=\det\left(  A+B+X\right)  -\det\left(  A+X\right)  -\det(B+X)+\det(X)\geq0
\]
for all $A,B,X\in\operatorname*{Sym}\nolimits^{+}(N,\mathbb{R}).$ This inequality is
mentioned in~\cite[Problem 36, pg.~215]{Zhang}. For $X=0$ it reduces to the property of superaditivity of the function $\det,$%
\[
\det(A+B)\geq\det(A)+\det(B).
\]

Our next goal is to show the more challenging third-order inequality%
\[
\Delta_{A}\left(  \Delta_{B}(\Delta_{C}\det\right)  )\geq0.
\]
We require some preparatory lemmas before stating our proof.

\begin{lemma}
  \label{lem:ek}
  Let $A, B, C \in \operatorname*{Sym}\nolimits^{+}(N,\mathbb{R})$, and let $e_k(X)$ denote the $k$-th elementary symmetric function of a matrix $X$ ($0\le k \le N$). Then,
  \begin{equation*}
    e_{k}(A)+e_{k}(B)+e_{k}(C)+e_{k}(A+B+C)\geq e_{k}(A+B)+e_{k}(B+C)+e_{k}(C+A).
  \end{equation*}
\end{lemma}
\begin{proof}
  Recall that for an $N\times N$ matrix $X$, we have $e_k(X) = \trace(\wedge^k X)$, where $\wedge$ denotes the anti-symmetric tensor product; moreover, $\wedge^k(X)=P_k^*(\otimes^k X)P_k$ for a suitable projection matrix $P_k$---see e.g.,~\cite[pg.~18]{Bhatia1997} for these facts. Using $P_k$ in Theorem 2.1 in~\cite{BS2015}, the claimed inequality for $e_k$ follows.
\end{proof}

\begin{lemma}
\label{thm3mon}If $A,B,C,X\in\operatorname*{Sym}\nolimits^{+}(N,\mathbb{R}),$
then%
\begin{multline*}
\det\left(  A+X\right)  +\det\left(  B+X\right)  +\det\left(  C+X\right)
+\det(A+B+C+X)\\
\geq\det\left(  A+B+X\right)  +\det\left(  B+C+X\right)  +\det\left(
C+A+X\right)  +\det X.
\end{multline*}

\end{lemma}

For $X=0$ we get the determinantal Hornich-Hlawka inequality%
\begin{multline*}
\det A+\det B+\det C+\det(A+B+C)\\
\geq\det\left(  A+B\right)  +\det\left(  B+C\right)  +\det\left(  C+A\right)
,
\end{multline*}
first noticed by Lin~\cite{Lin}, who provided a proof based on eigenvalue majorization.

\begin{proof}[\textbf{Proof.}]
Without loss of generality we may assume that $X$ is invertible (for example,
replace $X$ by $X+\varepsilon I$ if necessary). Then
\[
\det(A+X)=\det(X)\det(X^{-1/2}AX^{-1/2}+I),
\]
where $I$ is the identity matrix; the inequality under attention is then
equivalent to
\[%
\begin{split}
\det(A+I) &  +\det(B+I)+\det(C+I)+\det(A+B+C+I)\\
&  \geq\det(A+B+I)+\det(B+C+I)+\det(C+A+I)+1.
\end{split}
\]
Consider the function $f(A):=\det(A+I)-1$. Then the above inequality becomes
\begin{equation}
f(A)+f(B)+f(C)+f(A+B+C)\geq f(A+B)+f(B+C)+f(C+A).\label{eq:13}%
\end{equation}
Now recall the well-known expansion
\[
\det(A+I)=\sum_{k=0}^{N}e_{k}(A),
\]
where $e_{k}(\cdot)$ denotes the $k$-th elementary symmetric polynomial
($e_{0}=1,e_1=\trace...,~e_{N}=\det$)---see \cite[Theorem 7.1.2, pg.~197]{Mirsky} Thus,
$f(A)=\sum_{k=1}^{N}e_{k}(A)$, and inequality~\eqref{eq:13} becomes
\begin{multline*}
\sum_{k=1}^{N}\bigl[e_{k}(A)+e_{k}(B)+e_{k}(C)+e_{k}(A+B+C)\bigr]\\
\geq\sum_{k=1}^{N}\bigl[e_{k}(A+B)+e_{k}(B+C)+e_{k}(C+A)\bigr].
\end{multline*}
The proof ends by applying Lemma~\ref{lem:ek} for each $k\in\{1,...,N\}$.
\end{proof}

A similar argument using \cite[Corollary 3.4]{BS2015}  yields
that the function $\det$ has positive differences of any order.

\begin{theorem}
\label{thm_n-mon of det}We have%
\begin{align*}
&  \Delta_{A_{1}}\left(  \Delta_{A_{2}}(...(\Delta_{A_{n}}\det\right)  ))(X)\\
&  =\sum_{k=0}^{n}(-1)^{k+1}\sum_{1\leq i_{1}<i_{2}<\cdots<i_{k}\leq n}%
\det(A_{i_{1}}+\cdots+A_{i_{k}}+X)\geq0,
\end{align*}
whenever $A_{1},\ldots,A_{n},X\in\operatorname*{Sym}^{+}(N,\mathbb{R)}$ and
$n\geq1.$ In other words, the restriction of the $\det$ function to
$\operatorname*{Sym}^{+}(N,\mathbb{R)}$ has positive differences of any order
$n\geq0.$
\end{theorem}

Lemma \ref{thm3mon} can be extended to a larger class of matrix functions,
that of immanants. The \emph{immanant} \emph{function} $d_{\chi}^{G},$
associated to a subgroup $G$ of the symmetric group $\mathcal{S}_{N}$ of $N$
letters and to an irreducible character $\chi$ of $G,$ is defined via the
formula%
\[
d_{\chi}^{G}(A)=\sum\nolimits_{\sigma\in G}\chi(\sigma)\prod\nolimits_{i=1}%
^{N}a_{i~\sigma(i)},~~A\in\operatorname*{Sym}\nolimits^{+}(N,\mathbb{R)}.
\]

When $G=\mathfrak{G}_{m}$ and $\chi(\sigma)=\operatorname*{sgn}\sigma$ we have
$d_{\chi}^{G}(A)=\det A,$ while for $\chi(\sigma)\equiv1$ we obtain the
permanent of $A.$

The following two inequalities%
\begin{gather*}
d_{\chi}^{G}\left(  A+X\right)  -d_{\chi}^{G}(X)\geq0\\
d_{\chi}^{G}\left(  A+B+X\right)  -d_{\chi}^{G}\left(  A+X\right)  -d_{\chi
}^{G}(B+X)+d_{\chi}^{G}(X)\geq0
\end{gather*}
occur for all $A,B,X\in\operatorname*{Sym}\nolimits^{+}(N,\mathbb{R}).$ See
respectively Merris \cite{Merris}, p. 228 and Paksoy, Turkmen and Zhang
\cite{PTZ}. The fact that $d_{\chi}^{G}$ has positive differences of third
order on $\operatorname*{Sym}\nolimits^{+}(N,\mathbb{R})$ makes the objective
of the following result:

\begin{theorem}
\label{thm_imm}If $A,B,C,X\in\operatorname*{Sym}\nolimits^{+}(N,\mathbb{R)},$
then
\[%
\begin{split}
d_{\chi}^{G}(A+X)  &  +d_{\chi}^{G}(B+X)+d_{\chi}^{G}(C+X)+d_{\chi}%
^{G}(A+B+C+X)\\
&  \geq d_{\chi}^{G}(A+B+X)+d_{\chi}^{G}(B+C+X)+d_{\chi}^{G}(C+A+X)+d_{\chi
}^{G}(X).
\end{split}
\]

\end{theorem}

Using arguments from multilinear algebra one can show that there exists a
matrix $Z_{G,\chi}$ such that%
\begin{equation}
d_{\chi}^{G}(X)=Z_{G,\chi}^{\ast}\left(  \otimes^{N}X\right)  Z_{G,\chi
}.\label{repr_dG}%
\end{equation}
See \cite{Mar}, p. 126. The representation (\ref{repr_dG}) turns the assertion
of Theorem \ref{thm_imm} into an immediate consequence of the following
general result:

\begin{theorem}
\label{thm_OperatorHH}$($Operator Hornich-Hlawka inequality$)$. If
$A,B,C,X\in\operatorname*{Sym}\nolimits^{+}(N,\mathbb{R)}$ and $p\geq1$ an
integer, then
\begin{equation}%
\begin{split}
\otimes^{p}(A+X) &  +\otimes^{p}(B+X)+\otimes^{p}(C+X)+\otimes^{p}(A+B+C+X)\\
&  \geq\otimes^{p}(A+B+X)+\otimes^{p}(B+C+X)+\otimes^{p}(C+A+X)+\otimes^{p}X,
\end{split}
\label{ineq_operator HH}%
\end{equation}
in the L\"{o}wner order. Here $\otimes$ denotes the usual tensor product.
\end{theorem}

The proof of Theorem \ref{thm_OperatorHH} follows the proof structure of \cite[Theorem~2.1]{BS2015}, but due to the additional $X$ term in~\eqref{ineq_operator HH} it turns out to be more intricate and requires some preparation. We start by introducing the following convenient notation:
\begin{equation}
X^{j}\equiv\otimes^{j}X=X^{\otimes j},\ \ \text{for}\ \ j\geq1,\quad
\text{and}\ X^{0}=1\in\mathbb{N}.\label{eq:notation}%
\end{equation}
The slight abuse of notation $X^{0}=1$ will be helpful in simplifying the presentation.

\begin{lemma}
\label{lem_main}Let $k,l\geq0$ be integers and $X,Y,Z,V\in\operatorname*{Sym}%
\nolimits^{+}(N,\mathbb{R)}$. Then,%
\begin{multline}
X^{k}\otimes V\otimes X^{l}+(X+Y+Z)^{k}\otimes V\otimes(X+Y+Z)^{l}%
\label{eq:20}\\
\geq(X+Y)^{k}\otimes V\otimes(X+Y)^{l}+(X+Z)^{k}\otimes V\otimes
(X+Z)^{l},\nonumber
\end{multline}
in the sense of L\"{o}wner order.
\end{lemma}

\begin{proof}
The proof is by induction on $k$ and $l$; we provide the argument for $k$,
which holds essentially unchanged for $l$. This approach suffices since for a
fixed but arbitrary $l\geq0$ we prove that the result holds for all $k\geq0$;
similarly, for an arbitrarily fixed $k\geq0$ the result holds for all $l\geq0$.

For the base case of induction, suppose $k=0$. Then the inequality~under
attention reduces to the following one:
\begin{equation}
V\otimes X^{l}+V\otimes(X+Y+Z)^{l}\geq V\otimes(X+Y)^{l}+V\otimes
(X+Z)^{l}.\label{eq:5}%
\end{equation}
We know from \cite[Theorem~2.1]{BS2015} that
\begin{equation}
X^{l}+(X+Y+Z)^{l}\geq(X+Y)^{l}+(X+Z)^{l}.\label{eq:23}%
\end{equation}
Since tensor product preserves inequalities, taking tensor product with $V$ on
both sides of the inequality~\eqref{eq:23} one immediately get~\eqref{eq:5}.
Assume therefore that inequality~in the the statement of Lemma \ref{lem_main}
holds for a fixed $l$ and some $k>0$. Then, consider
\begin{multline*}
(X+Y+Z)^{k+1}\otimes V\otimes(X+Y+Z)^{l}\\
=(X+Y+Z)\otimes\left[  (X+Y+Z)^{k}\otimes V\otimes(X+Y+Z)^{l}\right]  \\
\geq(X+Y+Z)\otimes\left[  \left(  X+Y\right)  ^{k}\otimes V\otimes\left(
X+Y\right)  ^{l}\right.  \\
\left.  +\left(  X+Z\right)  ^{k}\otimes V\otimes\left(  X+Z\right)
^{l}-X^{k}\otimes V\otimes X^{l}\right]  ,\\
=(X+Y)^{k+1}\otimes V\otimes(X+Y)^{l}+(X+Z)^{k+1}\otimes V\otimes
(X+Z)^{l}-X^{k+1}\otimes V\otimes X^{l}+\mathcal{T},
\end{multline*}
where the inequality follows from the induction hypothesis and the elementary
monotonicity properties of the tensor product. It remains to show that the
term
\[
\mathcal{T}=Z\otimes(X+Y)^{k}\otimes V\otimes(X+Y)^{l}+Y\otimes(X+Z)^{k}%
\otimes V\otimes(X+Z)^{l}-(Y+Z)\otimes X^{k}\otimes V\otimes X^{l}%
\]
is a nonnegative operator. Since $X,Y,Z\geq0$, it follows that
$X+Y\geq X$ and $X+Z\geq X$. Thus, the positive terms in $\mathcal{T}$
attached to $Y$ and $Z$ are clearly bigger than the respective negative terms,
whence $\mathcal{T}\geq0$. Inductively, we can conclude that for fixed $l$,
the inequality in the statement of Lemma \ref{lem_main} holds for all $k\geq
0$. Applying a similar argument for $l$, we conclude that this inequality
works in full generality.
\end{proof}

We are now in a position to detail the proof of Theorem \ref{thm_OperatorHH}.

\begin{proof}
[\textbf{Proof of Theorem \ref{thm_OperatorHH}}]Using the auxiliary function
\[
f_{p}(Z)=(Z+X)^{p}-X^{p},
\]
the inequality of interest~(\ref{ineq_operator HH}) can be rewritten as
\[
f_{p}(A)+f_{p}(B)+f_{p}(C)+f_{p}(A+B+C)\geq f_{p}(A+B)+f_{p}(B+C)+f_{p}(C+A).
\]
Now introduce the function
\[
g_{p}(Z)=f_{p}(Z)+f_{p}(B)+f_{p}(C)+f_{p}(Z+B+C)-f_{p}(Z+B)-f_{p}%
(Z+C)-f_{p}(B+C).
\]
We will show that $g_{p}$ is monotonic (as a map from $\operatorname*{Sym}%
\nolimits^{+}(N,\mathbb{R)}$ into itself, under the L\"{o}wner order).
Once this monotonicity is established we can conclude that%
\[
g_{p}(A)\geq g_{p}(0)=0\text{\quad for every }A\in\operatorname*{Sym}%
\nolimits^{+}(N,\mathbb{R)},
\]
a fact equivalent to the assertion of Theorem~\ref{thm_OperatorHH}. 

Monotonicity of $g_{p}$ follows from Lemma~\ref{lemAmann} by considering its derivative. To that end, consider the (directional) derivative of the map $\Phi(Z)=Z^{p}$:
\begin{align*}
D(Z^{p})[V] &  =V\otimes Z\otimes\cdots\otimes Z+Z\otimes V\cdots\otimes
Z+\cdots+Z\otimes\cdots\otimes Z\otimes V\\
&  =\sum_{j=0}^{p-1}Z^{j}\otimes V\otimes Z^{p-1-j},
\end{align*}
whenever $Z\in\operatorname*{Sym}\nolimits^{++}(N,\mathbb{R)}$ and
$V\in\operatorname*{Sym}\nolimits^{+}(N,\mathbb{R)}.$ See \cite[Eq.~(2.13), pg.~44]{Bhatia2007}.  Indeed, applying this formula to $f_{p}(Z)=(Z+X)^{p}-X^{p}$
we obtain
\[
Df_{p}(Z)[V]=\sum_{j=0}^{p-1}(Z+V)^{j}\otimes V\otimes(Z+V)^{p-1-j},
\]
which in turn leads to the identity
\begin{multline*}
Dg_{p}(Z)[V]\\
=Df_{p}(Z)[V]+Df_{p}(Z+B+C)[V]-Df_{p}(Z+B)[V]-Df_{p}(Z+C)[V]\\
=\sum_{j=0}^{p-1}\left[  (Z+V)^{j}\otimes V\otimes(Z+V)^{p-1-j}\right.
+(Z+B+C+V)^{j}\otimes V\otimes(Z+B+C+V)^{p-1-j}\\
-(Z+B+V)^{j}\otimes V\otimes(Z+B+V)^{p-1-j}\left.  -(Z+C+V)^{j}\otimes
V\otimes(Z+C+V)^{p-1-j}\right].
\end{multline*}
But this sum evaluates to a positive quantity, which follows from Lemma~\ref{lem_main} upon
setting $k\leftarrow j$, $l\leftarrow p-1-j$ and $x\leftarrow x+j$. Now the
proof is complete.
\end{proof}

\section{Further comments and extensions}
\subsection{Multivariable case for positive operators}
Given $A_{1},...,A_{n},X\in\operatorname*{Sym}\nolimits^{+}(N,\mathbb{R)}$ and
$p\in\mathbb{N},$ let us consider the matrices
\begin{equation}
S_{0}^p=\otimes^{p}X,\text{\quad}S_{k}^p=\sum_{1\leq i_{1}<\cdots<i_{k}\leq
n}\otimes^{p}(A_{i_{1}}+\cdots+A_{i_{k}}+X)\text{ }(1\leq k\leq
N).\label{eq:8}%
\end{equation}
Then one can prove the following generalization of the operator Hornich-Hlawka inequality
(Theorem~\ref{thm_OperatorHH}):
\begin{equation}
S_{n}^p+S_{n-2}^p+\cdots\geq S_{n-1}^p+S_{n-3}^p+\cdots~,\label{eq:21}%
\end{equation}
Theorem~3.3 of~\cite{BS2015} proves inequality~\eqref{eq:21} for the special case $X=0$. One can prove the general case by a suitable monotonicity argument as in
Section~\ref{sec:cones}. The details are tedious so we omit them, leaving them as an exercise for the interested reader. 

\subsection{Two inequalities for determinants}
Serre~\cite{Serre} noticed that $\det^{1/2}$ viewed as a
function on $\operatorname*{Sym}\nolimits^{+}(2,\mathbb{R)}$ verifies the
opposite of the Hornich-Hlawka inequality,%
\begin{multline}
\label{eq:det1}
\det\nolimits^{1/2}A+\det\nolimits^{1/2}B+\det\nolimits^{1/2}C+\det
\nolimits^{1/2}(A+B+C)\\
\leq\det\nolimits^{1/2}\left(  A+B\right)  +\det\nolimits^{1/2}\left(
B+C\right)  +\det\nolimits^{1/2}\left(  C+A\right).
\end{multline}
While in~\cite{suvDetMO}, the second named author proved the following Minkowski-like inequality for $\det^{\frac1n}$, by viewing it as a function on $\operatorname*{Sym}\nolimits^{+}(n,\mathbb{R)}$:
\begin{multline}
\label{eq:det2}
    \det\nolimits^{\frac1n}(A+B)\det\nolimits^{\frac1n}(A+C)
    \ge
    \det\nolimits^{\frac1n}B
    \det\nolimits^{\frac1n}C + 
    \det\nolimits^{\frac1n}A \det\nolimits^{\frac1n}(A+B+C).
\end{multline}
Inequality~\eqref{eq:det2} is stronger the a usual log-supermodularity inequality for $\det$, given the extra $\det\nolimits^{\frac1n}B\det\nolimits^{\frac1n}C$ term on its right hand side.

\subsection{Vasic-Adamovic inequalities}
We close the paper by briefly mentioning the related class of Vasic-Adamovic inequalities. Inspired by the work of D.~M.~Smiley and M.~F.~Smiley~\cite{SmSm} on polygonal inequalities,\ P.~M.~Vasi\'{c} and D.~D.~Adamovi\'{c}~\cite{VA1968} found
an inductive scheme for generating inequalities for any function that verifies the
functional Hornich-Hlawka inequality. We state here a slightly modified version of their result~\cite[Theorem 2, pg.~528]{MPF}:
\begin{theorem}
\label{thmVA}Let $\mathcal{S}$ be a commutative additive semigroup with $0$ and $\mathcal{G}$ be an ordered abelian group (i.e., an abelian group with an order relation $\leq$) such that%
\[
x,y,z\in\mathcal{G}\text{ and }x\leq y\text{ implies }x+z\leq y+z.
\]
If $\varphi:\mathcal{S}\rightarrow\mathcal{G}$ is a function such that for all $x_{1},x_{2},x_{3}\in\mathcal{S}$, we have
\[
\sum\nolimits_{1\leq{i}<{j}\leq3}{\varphi\left(  {{x{_{i}+x}}}_{j}\right)  }%
\leq\sum\nolimits_{k=1}^{3}{\varphi\left(  x_{k}\right)  }+{\varphi}\Bigl(
{\sum\nolimits_{k=1}^{3}}x_{k}\Bigr)
\]
then for each pair $\left\{
k,n\right\}  $ of integers with $2\leq k<n$ we also have%
\[
\sum\limits_{1\leq{i_{1}}<...<{i_{k}}\leq n}{\varphi\Bigl(  {\sum\limits_{j=1}^{k}{x{_{i_{j}}}}}\Bigr)  }\leq\binom{n-2}{k-1}\sum
\limits_{k=1}^{n}{\varphi\left(  x_{k}\right)  }+\binom{n-2}{k-2}{\varphi
}\Bigl(  {\sum\limits_{k=1}^{n}}x_{k}\Bigr)  ,
\]
whenever $x_{1},...,x_{n}\in\mathcal{S}.$
\end{theorem}

This claim applies to every function $f$ (defined on a convex cone $\mathcal{C}$ and
taking values in an ordered Banach space $E)$ that has positive differences of
the third order. Indeed, in their case,
\begin{multline*}
f\left(  x+t\right)  +f\left(  y+t\right)  +f\left(  z+t\right)  +f\left(
x+y+z+t\right) \\
\geq f\left(  x+y+t\right)  +f\left(  y+z+t\right)  +f\left(  z+x+t\right)
+f(t)
\end{multline*}
for all points $x,y,z,t\in\mathcal{C}$ and performing the change of function
$\varphi(v)=f(v+t)-f(t)$ we obtain a function that verifies the Hornich-Hlawka
inequality $\;$%
\[
\varphi\left(  x\right)  +\varphi\left(  y\right)  +\varphi\left(  z\right)
+\varphi\left(  x+y+z\right)  \geq\varphi\left(  x+y\right)  +\varphi\left(
y+z\right)  +\varphi\left(  z+x\right)  .
\]


\end{document}